\documentclass[conference, 10pt, twocolumn]{ieeeconf}       
\bstctlcite{bstctl:etal, bstctl:nodash, bstctl:simpurl}
\def\BibTeX{{\rm B\kern-.05em{\sc i\kern-.025em b}\kern-.08em
    T\kern-.1667em\lower.7ex\hbox{E}\kern-.125emX}}
    
\usepackage[left=54pt, right=54pt,bottom=54pt, top=54pt]{geometry}

\usepackage{amsmath,mathrsfs,amsfonts,amssymb,graphicx,epsfig}
 \usepackage{amsthm}
\usepackage{subcaption}
\usepackage{color,multirow,rotating}
\usepackage{algorithm,algpseudocode,algorithmicx}
\usepackage{cite,url,framed,bm,balance,nicematrix,physics}
\usepackage{stmaryrd,mathtools}

\newtheorem{theorem}{Theorem}

\newtheorem{lemma}{Lemma}

\newtheorem{remark}{Remark}

\usepackage{tikz}
\usetikzlibrary{calc,trees,positioning,arrows,chains,shapes.geometric,decorations.pathreplacing,decorations.pathmorphing,shapes, matrix,shapes.symbols}

\newcommand{\hess}{{\mathrm{Hess}}}
\newcommand{\complexi}{{\mathrm{i}}}

\newcommand{\blue}{\color{blue}}

\makeatletter
\newcommand{\pushright}[1]{\ifmeasuring@#1\else\omit\hfill$\displaystyle#1$\fi\ignorespaces}

\newcommand{\dotminus}{\mathbin{\text{\@dotminus}}}

\newcommand{\@dotminus}{%
  \ooalign{\hidewidth\raise1ex\hbox{.}\hidewidth\cr$\m@th-$\cr}%
}

\allowdisplaybreaks

\title{\LARGE\textbf{
Control-affine Schr\"{o}dinger Bridge and Generalized Bohm Potential}
}

\author{Alexis M.H. Teter, Abhishek Halder, Michael D. Schneider, Alexx S. Perloff, Jane Pratt,\\Conor M. Artman, and Maria Demireva
\thanks{Teter is with the Department of Applied Mathematics, University of California Santa Cruz, CA 95064, USA, {\tt\footnotesize{amteter@ucsc.edu}}, and the Lawrence Livermore National Laboratory, Livermore, CA 94550, USA, {\tt\footnotesize{teter1@llnl.gov}}.%
}
\thanks{Halder is with the Department of Aerospace Engineering, Iowa State University, Ames, IA 50011, USA, {\tt\footnotesize{ahalder@iastate.edu}}.%
}
\thanks{Schneider, Perloff, Pratt, Artman, Demireva are with the Lawrence Livermore National Laboratory, Livermore, CA 94550, USA, {\tt\footnotesize{\{schneider42,perloff1,pratt34,artman1,demireva1\}\allowbreak @llnl.gov}}.%
}
\thanks{This work was performed under the auspices of the U.S. Department of Energy by Lawrence Livermore National Laboratory under Contract DE-AC52-07NA27344. Partial funding for this work was provided by LLNL Laboratory Directed Research and Development grant GS 25-ERD-044. Document release number: LLNL-JRNL-2008865.}
}

\IEEEoverridecommandlockouts
\begin{document}
\bstctlcite{IEEEexample:BSTcontrol}
\maketitle
\thispagestyle{empty}
\pagestyle{empty}

\begin{abstract}
From a stochastic control perspective, the Schr\"{o}dinger bridge is a density-valued continuous curve parametrized by time that connects a given pair of initial and terminal probability densities via minimum effort controlled Brownian motion. The control-affine Schr\"{o}dinger bridge extends this idea to a generic control-affine It\^{o} diffusion, possibly with an additive state cost. In this work, we recast the necessary conditions of optimality for the control-affine Schr\"{o}dinger bridge problem as a two point boundary value problem for a quantum mechanical Schr\"{o}dinger PDE with complex potential. This complex-valued potential is a generalization of the real-valued Bohm potential in quantum mechanics. Our derived potential is akin to the optical potential in nuclear physics where the real part of the potential encodes elastic scattering (transmission of wave function), and the imaginary part encodes inelastic scattering (absorption of wave function). The key takeaway is that the process noise that drives the evolution of probability densities induces an absorbing medium in the evolution of wave function. These results make new connections between control theory and non-equilibrium statistical mechanics through the lens of quantum mechanics.
\end{abstract}


\section{Introduction}\label{sec:introduction}
In 1931-32, Erwin Schr\"{o}dinger posed \cite{Sch31,Sch32} the question: what is the most likely probability density-valued continuous curve connecting two given probability density functions when the prior dynamics is Brownian motion? The diffusion process generating this curve is now known as the \emph{Schr\"{o}dinger bridge}, and has found widespread applications in stochastic control \cite{leonard2014survey,chen2021stochastic,caluya2021wasserstein,haddad2020prediction,nodozi2023neural,teter2025probabilistic} and generative AI \cite{de2021diffusion,shi2022conditional,liu2022deep,liu2023i2sb,xie2024bridging}.

In its original incarnation, the Schr\"{o}dinger bridge is a stochastic calculus of variations problem most naturally described through the theory of large deviations \cite{dembo2009large}, specifically by conditional Sanov's theorem \cite{csiszar1984sanov,dawson1990schrodinger,aebi1992large}. That the Schr\"{o}dinger bridge admits stochastic optimal control reformulation, was understood only in the late 20th century  \cite{follmer1988random,dai1991stochastic,pavon1991free,blaquiere1992controllability}.

This letter is similar in spirit to Schr\"{o}dinger's original motivation: to find points of contact between quantum mechanics and non-equilibrium statistical mechanics. In particular, we focus on a variant called the \emph{control-affine Schr\"{o}dinger bridge} (\texttt{caSB}) \cite{teter2025hopf} that 
concerns with the following stochastic optimal control problem over a given time horizon $[t_0,t_1]$:
\begin{subequations}
\begin{align}
&\underset{\left(\rho^{\bm{u}},\bm{u}\right)\in\mathcal{P}_{01}\times\mathcal{U}}{\arg\inf}\int_{t_0}^{t_1}\mathbb{E}_{\rho^{\bm{u}}}\left[q\left(t,\bm{x}_{t}^{\bm{u}}\right) + \frac{1}{2}\|\bm{u}\|_2^2\right]\differential t
\label{CASBPobjective}\\
&\text{subject to}\quad\partial_t\rho^{\bm{u}}
+ \nabla_{\bm{x}_t^{\bm{u}}}\cdot\left(\rho^{\bm{u}}\left(\bm{f}\left(t,\bm{x}_{t}^{\bm{u}}\right)+\bm{g}\left(t,\bm{x}_{t}^{\bm{u}}\right)\bm{u}\right)\right) \nonumber\\
&\hspace*{1in}= \frac{1}{2}\Delta_{\bm{\Sigma}\left(t,\bm{x}_{t}^{\bm{u}}\right)}\:\rho^{\bm{u}},\label{CASBPpdeconstraint}
\end{align}
\label{CASBP}
\end{subequations}
where $q(\cdot)$ is some bounded measurable state cost, $\mathcal{P}_{01} := \{t\mapsto \rho(t,\cdot)\:\text{continuous} \mid \rho \geq 0, \int \rho(t,\cdot) = 1\:\forall t\in[t_0,t_1], \rho(t=t_0,\cdot) = \rho_0(\cdot), \rho(t=t_1,\cdot) = \rho_1(\cdot)\}$, the weighted Laplacian\footnote{The case $\bm{\Sigma}=\bm{I}$ gives standard Laplacian $\Delta_{\bm{x}}:=\sum_{i,j=1}^{n}\partial^{2}_{x_{i}x_{j}}$.}
\begin{align}
\Delta_{\bm{\Sigma}(t,\bm{x})}\:\rho := \sum_{i,j=1}^{n}\partial^{2}_{x_{i}x_{j}}\left(\bm{\Sigma}_{ij}(t,\bm{x})\rho(t,\bm{x})\right),
\label{DefWeightedLaplacian}    
\end{align}
and $\nabla_{\bm{x}}$ denotes the standard Euclidean gradient w.r.t. vector $\bm{x}$. In particular, the composition $\nabla_{\bm{x}}\circ\nabla_{\bm{x}}=\hess_{\bm{x}}$, the Hessian. The inner product $\langle\nabla_{\bm{x}},\nabla_{\bm{x}}\rangle = \Delta_{\bm{x}}$, the standard Laplacian.

In \eqref{CASBP}, $\rho^{\bm{u}}$ is the probability density function (PDF) for the stochastic state $\bm{x}_t^{\bm{u}}\in\mathbb{R}^{n}$ that follows the control-affine It\^{o} stochastic differential equation:
\begin{align}
\differential\bm{x}_{t}^{\bm{u}} = \left(\bm{f}\left(t,\bm{x}_{t}^{\bm{u}}\right) + \bm{g}\left(t,\bm{x}_{t}^{\bm{u}}\right)\bm{u}\right)\differential t + \bm{\sigma}(t,\bm{x}_{t}^{\bm{u}})\differential\bm{w}_t,
\label{ControlledSDE}    
\end{align}
and the diffusion tensor $\bm{\Sigma}:=\bm{\sigma}\bm{\sigma}^{\top}\succeq\bm{0}$. In \eqref{ControlledSDE}, the $\bm{w}_t\in\mathbb{R}^{p}$ is standard Brownian motion, and the control $\bm{u}\in\mathcal{U}:=\{\bm{v}:[t_0,t_1]\times\mathbb{R}^{n}\mapsto\mathbb{R}^{m} \mid \|\bm{v}\|_2^2 < \infty\}$, the collection of Markovian finite energy inputs.

The given data for problem \eqref{CASBP} are: the time horizon $[t_0,t_1]$, the endpoint state PDFs $\rho_0,\rho_1$ having finite second moments, the bounded measurable state cost $q(\cdot)$, and the $\bm{f},\bm{g},\bm{\sigma}$ in \eqref{ControlledSDE} satisfying the standard assumptions:
\begin{itemize}
\item[\textbf{A1.}] $\exists c_1,c_2>0$ such that $\forall\bm{x},\bm{y}\in\mathbb{R}^n$, $\forall t\in[t_0,t_1]$, 
\begin{align*}
\|\bm{f}(t,\bm{x})\|_2 + \|\bm{\sigma}\left(t,\bm{x}\right)\|_{2} 
&\leq 
c_1\left(1 + \|\bm{x}\|_2\right), \\
\|\bm{f}(t,\bm{x}) - \bm{f}(t,\bm{y})\|_2 
&\leq 
c_2 \|\bm{x}-\bm{y}\|_2.    
\end{align*}
\item[\textbf{A2.}] $\exists c_3>0$ such that $\forall \bm{x}\in\mathbb{R}^{n}$, $\forall t\in[t_0,t_1]$,
\begin{align*}
    \langle \bm{x}, \bm{\Sigma}(t,\bm{x})\bm{x}\rangle 
    \geq 
    c_3 \|\bm{x}\|_2^2.
\end{align*}  
\end{itemize}
Solving \eqref{CASBP} amounts to minimizing an expected cost-to-go while transferring the controlled state between given endpoint state PDFs $\rho_0,\rho_1$ subject to the control-affine dynamics \eqref{ControlledSDE} and hard deadline constraints. Schr\"{o}dinger's original setting, referred simply as the Schr\"{o}dinger bridge (\texttt{SB}), is a special case of \eqref{CASBP}: $q\equiv 0$, $\bm{f}\equiv\bm{0}$, $\bm{g}=\bm{\sigma}=\bm{I}$.

From a control-theoretic perspective, the \texttt{caSB} is of more interest than the \texttt{SB}. This is because for control systems of practical interest, the unforced dynamics often have nontrivial prior drift $\bm{f}$ and diffusion coefficient $\bm{\sigma}$, as opposed to standard Brownian motion. Also, practical control systems have limited control authority encoded by the input coefficient $\bm{g}$. The state cost $q$ in \eqref{CASBPobjective} regularizes the optimally controlled sample paths for all $t\in[t_0,t_1]$ beyond minimum effort steering between the given endpoint statistics. Therefore, a better understanding of \eqref{CASBP} is needed.

Under the stated assumptions on the problem data (i.e., \textbf{A1}-\textbf{A2}, the PDFs $\rho_0,\rho_1$ have finite second moments, the state cost $q(\cdot)$ is bounded measurable), the existence-uniqueness for the solution of \eqref{CASBP} is guaranteed; see e.g., \cite{dawson1990schrodinger,leonard2011stochastic}. Let $S\in\mathcal{C}^{1,2}\left([t_0,t_1];\mathbb{R}^{n}\right)$ be the dual variable (value function) associated with the variational problem \eqref{CASBP}. Standard computation \cite[Thm. 1]{teter2025hopf} shows that the primal-dual pair $(\rho^{\bm{u}}_{\mathrm{opt}},S)$ for problem \eqref{CASBP} solves the following system of coupled nonlinear PDEs: 
\begin{subequations}
\begin{align}
&\text{(primal PDE)}\quad\partial_{t}\rho^{\bm{u}}_{\mathrm{opt}} + \nabla_{\bm{x}}\cdot\left(\rho^{\bm{u}}_{\mathrm{opt}}\left(\bm{f}+\bm{g}\bm{g}^{\top}\nabla_{\bm{x}}S\right)\right)\nonumber\\
&\hspace*{2in}= \frac{1}{2}\Delta_{\bm{\Sigma}}\:\rho^{\bm{u}}_{\mathrm{opt}},
\label{PrimalPDE}\\
&\text{(dual PDE)}\quad
\partial_{t} S +\langle\nabla_{\bm{x}}S,\bm{f}\rangle +\frac{1}{2}\langle\nabla_{\bm{x}}S,\bm{gg}^{\top}\nabla_{\bm{x}}S\rangle \nonumber\\
&\hspace*{1.5in}+\frac{1}{2}\langle\bm{\Sigma},\hess_{\bm{x}}S\rangle = q,\label{DualPDE}\\
&\text{(boundary conditions)}\quad \rho^{\bm{u}}_{\mathrm{opt}}\left(t=t_0,\cdot\right)=\rho_0(\cdot), \nonumber\\
&\hspace*{1.38in}\rho^{\bm{u}}_{\mathrm{opt}}\left(t=t_1,\cdot\right)=\rho_1(\cdot).\label{PrimalBC}
\end{align}
\label{FOCO}    
\end{subequations}
The optimal control is
\begin{align}
\bm{u}_{\mathrm{opt}} = \bm{g}^{\top}\nabla_{\bm{x}}S.
\label{OptimalControl}    
\end{align}
Notice that \eqref{PrimalPDE} is a controlled Fokker-Planck-Kolmogorov PDE in primal variable $\rho^{\bm{u}}_{\mathrm{opt}}$ while \eqref{DualPDE} is a Hamilton-Jacobi-Bellman (HJB) PDE in dual variable $S$. 

Using the Hopf-Cole transform $\left(\rho^{\bm{u}}_{\mathrm{opt}},S\right)\mapsto\left(\varphi,\widehat{\varphi}\right):= \left(\rho^{\bm{u}}_{\mathrm{opt}}\exp\left(-S/\lambda\right),\exp\left(S/\lambda\right)\right)$, $\lambda>0$, the system \eqref{FOCO} can be transformed \cite[Sec. IV-C]{teter2025hopf} into a system of PDEs in real-valued function pair $\left(\varphi,\widehat{\varphi}\right)$ such that $\rho^{\bm{u}}_{\mathrm{opt}}(t,\cdot) = \varphi(t,\cdot)\widehat{\varphi}(t,\cdot)$. This system is amenable for contractive fixed point recursions, facilitating a numerical solution for \eqref{FOCO}. 

While the focus of \cite{teter2025hopf} was on understanding the benefits of the Hopf-Cole transform to \eqref{FOCO}, here we apply a different transform on \eqref{FOCO} to transcribe it to a quantum mechanical Schr\"{o}dinger PDE boundary value problem (BVP) in unknown wave function. Specifically, let $\complexi:=\sqrt{-1}$, use the superscript $^{\dagger}$ to denote complex conjugate, and let
\begin{align}
R := \frac{1}{2}\log\rho^{\bm{u}}_{\mathrm{opt}}.
\label{defR}
\end{align}
For fixed $\lambda>0$, we apply the Madelung transform $\left(\rho^{\bm{u}}_{\mathrm{opt}},S\right)\mapsto\left(\psi,\psi^{\dagger}\right)$, equivalently $\left(R,S\right)\mapsto\left(\psi,\psi^{\dagger}\right)$ \cite[Sec. 7]{nagasawa1989transformations}, given by
\begin{subequations}
\begin{align}
\psi &:= \exp\left(R+\frac{\complexi}{\lambda}S\right), \label{defpsi}\\
\psi^{\dagger} &:= \exp\left(R-\frac{\complexi}{\lambda}S\right). \label{defpsidagger}
\end{align}
\label{NagasawaTransform}
\end{subequations}
We refer to $\psi$ as the \emph{wave function}, and $\psi^{\dagger}$ as the \emph{conjugate wave function}. Note that \eqref{defR}-\eqref{NagasawaTransform} imply \emph{Born's relation} \cite{born1926quantenmechanik}:
\begin{align}
\rho^{\bm{u}}_{\mathrm{opt}}(t,\cdot) = \psi(t,\cdot)\psi^{\dagger}(t,\cdot) \quad\forall t\in[t_0,t_1],
\label{psitimespsidagger}\end{align}
    giving a complex-valued factorization of $\rho^{\bm{u}}_{\mathrm{opt}}$. Because of complex conjugacy, it suffices to derive a single PDE BVP for the transformed variable $\psi(t,\cdot)$, unlike the case for the Hopf-Cole transform. 

At our level of generality (i.e., for the \texttt{caSB}) it is not immediately clear whether the transformed PDE BVP can be related to the quantum mechanical Schr\"{o}dinger PDE. Even if this is possible, it is unclear what the structure would be for the corresponding quantum potential.  

\subsubsection*{Contributions}
\begin{itemize}
    \item We prove that the aforementioned transformed PDE BVP is in the form
    \begin{subequations}
\begin{align}
&\complexi\lambda\partial_{t}\psi = -\dfrac{\lambda^2}{2}\Delta_{\bm{\Sigma}}\psi + V_{\texttt{caSB}}\psi,\label{WaveFunctionPDEgeneralized}\\
&\psi(t_{0},\bm{x})\psi^{\dagger}(t_{0},\bm{x}) = \rho_{0}, \: \psi(t_{1},\bm{x})\psi^{\dagger}(t_{1},\bm{x}) = \rho_{1},
\label{WaveFunctionBCgeneralized}
\end{align}
\label{WaveFunctionBVPgeneralized}
\end{subequations}
where the quantum potential $V_{\texttt{caSB}}$ is complex!

\item We show that in the case of \texttt{SB}, \eqref{WaveFunctionPDEgeneralized} reduces to the more familiar Schr\"{o}dinger PDE form:
\begin{align}
\complexi\partial_{t}\psi = -\dfrac{1}{2}\Delta_{\bm{x}}\psi + V_{\texttt{SB}}\psi,\label{WaveFunctionPDEclassical}
\end{align}
with \eqref{WaveFunctionBCgeneralized} unchanged, where the potential $V_{\texttt{SB}}$ is still complex. The derived $V_{\texttt{SB}}$ is a considerable generalization of the Bohm potential \cite{bohm1952suggested,bohm1954model}. We explain why the imaginary part of $V_{\texttt{SB}}$ cannot vanish, and that it implies an effective absorption for the wave function. In this sense, our results can be seen as a stochastic control-theoretic generalization of the deterministic de Broglie–Bohm theory.

\item The conceptual significance of \eqref{WaveFunctionBVPgeneralized} is to establish an equivalence between optimal density steering and wave steering for an important class of stochastic control systems given by \eqref{ControlledSDE}. This should be of broad interest. We stress here that the potentials $V_{\texttt{caSB}}$ in \eqref{WaveFunctionPDEgeneralized} and $V_{\texttt{SB}}$ in \eqref{WaveFunctionPDEclassical}, depend on $\psi$ via $R,S$. So \eqref{WaveFunctionPDEgeneralized} and \eqref{WaveFunctionPDEclassical} are nonlinear Schr\"{o}dinger PDEs. 


\end{itemize}

\subsubsection*{Related Works}
The connections between the Schr\"{o}dinger bridge and the Schr\"{o}dinger PDE in quantum mechanics were explored by Nagasawa \cite[Sec. 7]{nagasawa1989transformations}, where a Schr\"{o}dinger process was derived from the Schr\"{o}dinger PDE. Similar connections were pursued by Guerra and Morato \cite{guerra1983quantization}. Also starting from the Schr\"{o}dinger PDE, Ohsumi derived \cite{ohsumi2019interpretation} a stochastic control problem in the spirit of inverse optimal control. However, the resulting formulation is not a Schr\"{o}dinger bridge. More broadly, several works \cite{nelson1966derivation,yasue1981quantum,rosenbrock2002doing,teter2024weyl} have discussed connections between non-quantum stochastic control and quantum mechanics.

In contrast to the existing works, our developments in this letter start from the generic control-affine Schr\"{o}dinger bridge--a concrete stochastic optimal control problem--and from there, derive suitable versions of the Schr\"{o}dinger PDE.

We clarify here that our results concern with transforming the non-quantum Schr\"{o}dinger bridge (as in classical Markov diffusion process) to the Schr\"{o}dinger PDE. We do not study the quantum Schr\"{o}dinger bridge \cite{pavon2002quantum,pavon2010schrodinger}.



\section{Main Results}\label{sec:MainResults}

\subsection{Dynamics of the wave function for \texttt{caSB}}\label{subsec:psiCASB}
We start with the following Lemma.
\begin{lemma}[Weighted Laplacian of $R$]\label{Lemma:FromWeightedLaplacianOfRhoToWeightedLaplacianOfR}
For $R$ defined as in \eqref{defR}, we have
\begin{align}
\frac{1}{4\rho^{\bm{u}}_{\mathrm{opt}}} \Delta_{\bm{\Sigma}}\rho^{\bm{u}}_{\mathrm{opt}} &= \frac{1}{2}\Delta_{\bm{\Sigma}} R + \left(\nabla_{\bm{x}}R\right)^{\top}\bm{\Sigma}\nabla_{\bm{x}}R \nonumber \\ & \hspace{1cm}+ \left(\frac{1}{4}-\frac{1}{2}R\right)\langle\hess_{\bm{x}},\bm{\Sigma}\rangle.
\label{FromWeightedLaplacianOfRhoToWeightedLaplacianOfR}    
\end{align}    
\end{lemma} 
\begin{proof}
Using \cite[Lemma 1, eq. (16)]{teter2025hopf}, we have\footnote{The divergence of a matrix field $\bm{\Sigma}$, denoted as $\nabla_{\bm{x}}\cdot\bm{\Sigma}$, is understood as a vector with elements
$\left(\nabla_{\bm{x}}\cdot\bm{\Sigma}\right)_{i} := \displaystyle\sum_{j=1}^{n}\dfrac{\partial\Sigma_{ij}}{\partial x_{j}} \quad\forall i\in\{1,\hdots,n\}$.}
\begin{align}
\Delta_{\bm{\Sigma}} R &= R\langle\hess_{\bm{x}},\bm{\Sigma}\rangle + \langle\bm{\Sigma},\hess_{\bm{x}}R\rangle + 2\langle\nabla_{\bm{x}}\cdot\bm{\Sigma},\nabla_{\bm{x}}R\rangle\nonumber\\
\Rightarrow\quad& \frac{1}{2}\langle\bm{\Sigma},\hess_{\bm{x}}R\rangle + \langle\nabla_{\bm{x}}\cdot\bm{\Sigma},\nabla_{\bm{x}}R\rangle \nonumber \\ & \hspace{2cm} = \frac{1}{2}\Delta_{\bm{\Sigma}} R - \frac{1}{2}R\langle\hess_{\bm{x}},\bm{\Sigma}\rangle.
\label{WeightedLaplacianRexpansion}
\end{align}
By the same \cite[Lemma 1, eq. (16)]{teter2025hopf}, we also have
\begin{align}
\frac{1}{4\rho^{\bm{u}}_{\mathrm{opt}}}\Delta_{\bm{\Sigma}} \rho^{\bm{u}}_{\mathrm{opt}} &= \frac{1}{4}\langle\hess_{\bm{x}},\bm{\Sigma}\rangle + \frac{1}{4\rho^{\bm{u}}_{\mathrm{opt}}}\langle\bm{\Sigma},\hess_{\bm{x}}\rho^{\bm{u}}_{\mathrm{opt}}\rangle \nonumber \\ & \hspace{1cm} + \frac{1}{2\rho^{\bm{u}}_{\mathrm{opt}}}\langle\nabla_{\bm{x}}\cdot\bm{\Sigma},\nabla_{\bm{x}}\rho^{\bm{u}}_{\mathrm{opt}}\rangle.
\label{WeightedLaplacianRhoexpansion}
\end{align}
To express the RHS of \eqref{WeightedLaplacianRhoexpansion} in terms of $R$, notice that the first summand in this RHS is independent of $R$, and thanks to \eqref{defR}, the last summand is $\langle\nabla_{\bm{x}}\cdot\bm{\Sigma},\nabla_{\bm{x}}R\rangle$. For the middle summand in the RHS of \eqref{WeightedLaplacianRhoexpansion}, we find
\begin{align*}
\hess_{\bm{x}}R &= \hess_{\bm{x}}\left(\frac{1}{2}\log\rho^{\bm{u}}_{\mathrm{opt}}\right) \nonumber\\
&= -\dfrac{1}{2}\left(\frac{\nabla_{\bm{x}}\rho^{\bm{u}}_{\mathrm{opt}}}{\rho^{\bm{u}}_{\mathrm{opt}}}\right)\left(\frac{\nabla_{\bm{x}}\rho^{\bm{u}}_{\mathrm{opt}}}{\rho^{\bm{u}}_{\mathrm{opt}}}\right)^{\top} \!\!\!+ \dfrac{1}{2\rho^{\bm{u}}_{\mathrm{opt}}}\hess_{\bm{x}}\rho^{\bm{u}}_{\mathrm{opt}}\nonumber\\
&= -\dfrac{1}{2}\left(2\nabla_{\bm{x}}R\right)\left(2\nabla_{\bm{x}}R\right)^{\top} \!\!+ \dfrac{1}{2\rho^{\bm{u}}_{\mathrm{opt}}}\hess_{\bm{x}}\rho^{\bm{u}}_{\mathrm{opt}},
\end{align*}
which yields
\begin{align}
\dfrac{1}{4\rho^{\bm{u}}_{\mathrm{opt}}}\hess_{\bm{x}}\rho^{\bm{u}}_{\mathrm{opt}} &= \left(\nabla_{\bm{x}}R\right)\left(\nabla_{\bm{x}}R\right)^{\top} + \frac{1}{2}\hess_{\bm{x}}R.
\label{HessOfR}    
\end{align}
Hence, we can rewrite \eqref{WeightedLaplacianRhoexpansion} as
\begin{align}
\frac{1}{4\rho^{\bm{u}}_{\mathrm{opt}}}\Delta_{\bm{\Sigma}} \rho^{\bm{u}}_{\mathrm{opt}} &= \frac{1}{4}\langle\hess_{\bm{x}},\bm{\Sigma}\rangle + \left(\nabla_{\bm{x}}R\right)^{\top}\bm{\Sigma}\left(\nabla_{\bm{x}}R\right) \nonumber \\ & \hspace{1cm} + \frac{1}{2}\langle\bm{\Sigma},\hess_{\bm{x}}R\rangle + \langle\nabla_{\bm{x}}\cdot\bm{\Sigma},\nabla_{\bm{x}}R\rangle\nonumber\\
&= \frac{1}{4}\langle\hess_{\bm{x}},\bm{\Sigma}\rangle + \left(\nabla_{\bm{x}}R\right)^{\top}\bm{\Sigma}\left(\nabla_{\bm{x}}R\right) \nonumber \\& \hspace{1cm} + \frac{1}{2}\Delta_{\bm{\Sigma}} R - \frac{1}{2}R\langle\hess_{\bm{x}},\bm{\Sigma}\rangle,
\label{WeightedLaplacianRhoexpansionSimplified}
\end{align}
using \eqref{WeightedLaplacianRexpansion}. Grouping the first and the last summand in the RHS of \eqref{WeightedLaplacianRhoexpansionSimplified}, we obtain \eqref{FromWeightedLaplacianOfRhoToWeightedLaplacianOfR}.
\end{proof}

Using Lemma \ref{Lemma:FromWeightedLaplacianOfRhoToWeightedLaplacianOfR}, we next derive \eqref{WaveFunctionBVPgeneralized}.

\begin{theorem}\label{Thm:caSB}
Let $\psi$ be as in \eqref{defpsi}. Then \eqref{FOCO} is equivalent to \eqref{WaveFunctionBVPgeneralized} with complex potential $V_{\texttt{caSB}}$ having real ($\Re$) and imaginary ($\Im$) parts
\begin{align}
&\Re\left(V_{\texttt{caSB}}\right) = \dfrac{\lambda^2}{2}\langle\hess_{\bm{x}},\bm{\Sigma}\rangle +\dfrac{\lambda^2}{2}\langle\bm{\Sigma},\hess_{\bm{x}}R\rangle + \langle\nabla_{\bm{x}}S,\bm{f}\rangle\nonumber \\ &\hspace{0.5cm}+\dfrac{\lambda^2}{2}\|\nabla_{\bm{x}}R\|_{\bm{\Sigma}}^2 -\dfrac{1}{2}\|\nabla_{\bm{x}}S\|_{\bm{\Sigma}}^2+\lambda^2\langle\nabla_{\bm{x}}\cdot\bm{\Sigma},\nabla_{\bm{x}}R\rangle \nonumber\\
&\hspace{0.5cm} +\!\frac{1}{2}\langle\nabla_{\bm{x}}S,\bm{gg}^{\top}\nabla_{\bm{x}}S\rangle \!+\frac{1}{2}\langle\bm{\Sigma},\hess_{\bm{x}}S\rangle\! - q, \label{Re_V}\\
&\Im\left(V_{\texttt{caSB}}\right) = \lambda\bigg\{\dfrac{1}{2}\langle\bm{\Sigma},\hess_{\bm{x}}S\rangle + \left(\nabla_{\bm{x}}R\right)^{\top}\bm{\Sigma}\left(\nabla_{\bm{x}}S\right) \nonumber \\ &\hspace{0.5cm}+ \langle\nabla_{\bm{x}}\cdot\bm{\Sigma},\nabla_{\bm{x}}S\rangle -\langle\nabla_{\bm{x}}R,\bm{f}+\bm{gg}^{\top}\nabla_{\bm{x}}S\rangle\nonumber\\
&\hspace{0.5cm}-\frac{1}{2}\nabla_{\bm{x}}\cdot\left(\bm{f}+\bm{gg}^{\top}\nabla_{\bm{x}}S\right) + \frac{1}{2}\Delta_{\bm{\Sigma}} R \nonumber \\
&\hspace{0.5cm}+ \left(\nabla_{\bm{x}}R\right)^{\top}\bm{\Sigma}\nabla_{\bm{x}}R + \left(\frac{1}{4}-\frac{1}{2}R\right)\langle\hess_{\bm{x}},\bm{\Sigma}\rangle\bigg\},
\label{Im_V}
\end{align}
where the squared weighted norm $\|\cdot\|_{\bm{\Sigma}}^2 := \left(\cdot\right)^{\top}\bm{\Sigma}(\cdot)$.
\end{theorem}

\begin{proof}
To derive a PDE for $\psi$, we combine Lemma \ref{Lemma:FromWeightedLaplacianOfRhoToWeightedLaplacianOfR} with the primal-dual PDEs \eqref{PrimalPDE}-\eqref{DualPDE}, to obtain
\begin{align}
&\dfrac{1}{\psi}\partial_{t}\psi = \partial_{t}R + \dfrac{\complexi}{\lambda}\partial_{t}S\nonumber\\
&= \dfrac{1}{2\rho^{\bm{u}}_{\mathrm{opt}}}\left(-\nabla_{\bm{x}}\cdot\left(\rho^{\bm{u}}_{\mathrm{opt}}\left(\bm{f}+\bm{gg}^{\top}\nabla_{\bm{x}}S\right)\right) + \dfrac{1}{2}\Delta_{\bm{x}}\rho^{\bm{u}}_{\mathrm{opt}}\right) \nonumber\\ 
&+ \dfrac{\complexi}{\lambda}\!\left(\!\!-\langle\nabla_{\bm{x}}S,\bm{f}\rangle -\frac{1}{2}\langle\nabla_{\bm{x}}S,\bm{gg}^{\top}\nabla_{\bm{x}}S\rangle \nonumber \right. \\ & \left. \hspace{4cm}-\frac{1}{2}\langle\bm{\Sigma},\hess_{\bm{x}}S\rangle + q \right)\nonumber\\
&= -\langle\nabla_{\bm{x}}R,\bm{f}+\bm{gg}^{\top}\nabla_{\bm{x}}S\rangle -\frac{1}{2}\nabla_{\bm{x}}\cdot\left(\bm{f}+\bm{gg}^{\top}\nabla_{\bm{x}}S\right) \nonumber \\ &+ \frac{1}{2}\Delta_{\bm{\Sigma}} R + \left(\nabla_{\bm{x}}R\right)^{\top}\bm{\Sigma}\nabla_{\bm{x}}R + \left(\frac{1}{4}-\frac{1}{2}R\right)\langle\hess_{\bm{x}},\bm{\Sigma}\rangle \nonumber\\
&+ \dfrac{\complexi}{\lambda}\!\left(-\langle\nabla_{\bm{x}}S,\bm{f}\rangle -\frac{1}{2}\langle\nabla_{\bm{x}}S,\bm{gg}^{\top}\nabla_{\bm{x}}S\rangle \right. \nonumber\\ 
&\hspace{4cm}\left.-\frac{1}{2}\langle\bm{\Sigma},\hess_{\bm{x}}S\rangle + q\right).
\label{partiallogpsipartialtgeneralized}
\end{align}
We re-write \eqref{partiallogpsipartialtgeneralized} as
\begin{align}
\complexi\lambda \partial_{t}\psi = &\complexi \lambda \psi \left({\blue{-}}\langle\nabla_{\bm{x}}R,\bm{f}+\bm{gg}^{\top}\nabla_{\bm{x}}S\rangle + \left(\nabla_{\bm{x}}R\right)^{\top}\bm{\Sigma}\nabla_{\bm{x}}R\right. \nonumber \\ &\hspace{1.5cm}\left.-\frac{1}{2}\nabla_{\bm{x}}\cdot\left(\bm{f}+\bm{gg}^{\top}\nabla_{\bm{x}}S\right) + \frac{1}{2}\Delta_{\bm{\Sigma}} R \right. \nonumber \\ & \hspace{1.5cm}\left. + \left(\frac{1}{4}-\frac{1}{2}R\right)\langle\hess_{\bm{x}},\bm{\Sigma}\rangle \!\! \right) \nonumber\\
&-\psi\!\left(\!\!-\langle\nabla_{\bm{x}}S,\bm{f}\rangle -\frac{1}{2}\langle\nabla_{\bm{x}}S,\bm{gg}^{\top}\nabla_{\bm{x}}S\rangle \right. \nonumber \\ & \left. \hspace{1.5cm} -\frac{1}{2}\langle\bm{\Sigma},\hess_{\bm{x}}S\rangle + q\!\!\right).
\label{LHSOfSchrodingerPDEgeneralized}
\end{align}
We will come back to \eqref{LHSOfSchrodingerPDEgeneralized} in a bit.

From \cite[Lemma 1, eq. (16)]{teter2025hopf},
\begin{align}
\Delta_{\bm{\Sigma}}\psi &= \psi\langle\hess_{\bm{x}},\bm{\Sigma}\rangle + \langle\bm{\Sigma},\hess_{\bm{x}}\psi\rangle 
+ 2\langle\nabla_{\bm{x}}\cdot\bm{\Sigma},\nabla_{\bm{x}}\psi\rangle.
\label{Psi_Weighted_Laplacian}
\end{align}
Taking the gradient of 
\begin{align}
    \log\psi = R + \dfrac{\complexi}{\lambda}S, 
    \label{LogPsi}
\end{align}
we get $\nabla_{\bm{x}}\psi = \psi \left( \nabla_{\bm{x}}R + \dfrac{\complexi}{\lambda}\nabla_{\bm{x}}S\right)$, and re-write \eqref{Psi_Weighted_Laplacian} as
\begin{align}
    \Delta_{\bm{\Sigma}}\psi &=\psi\langle\hess_{\bm{x}},\bm{\Sigma}\rangle + \langle\bm{\Sigma},\hess_{\bm{x}}\psi\rangle \nonumber \\  &\hspace{1cm}+ 2\psi\bigg\langle\nabla_{\bm{x}}\cdot\bm{\Sigma},\nabla_{\bm{x}}R + \dfrac{\complexi}{\lambda}\nabla_{\bm{x}}S\bigg\rangle.
\label{Weighted_LaplacianPsi}    
\end{align}
Likewise, we use \eqref{LogPsi} to express $\langle\bm{\Sigma},\hess_{\bm{x}}\psi\rangle$ in terms of $R,S$ by first writing $$\hess_{\bm{x}}\log\psi = -\dfrac{\nabla_{\bm{x}}\psi}{\psi}\left(\dfrac{\nabla_{\bm{x}}\psi}{\psi}\right)^{\top} + \dfrac{1}{\psi}\hess_{\bm{x}}\psi.$$ 
Substituting \eqref{LogPsi} in above and rearranging yields
\begin{align*}
&\hess_{\bm{x}}\psi = \psi\bigg\{\hess_{\bm{x}}R + \dfrac{\complexi}{\lambda}\hess_{\bm{x}}S + \left(\nabla_{\bm{x}}R\right)\left(\nabla_{\bm{x}}R\right)^{\top} \nonumber \\ &\hspace{2cm}- \dfrac{1}{\lambda^2}\left(\nabla_{\bm{x}}S\right)\left(\nabla_{\bm{x}}S\right)^{\top} + \dfrac{\complexi}{\lambda}\left(\nabla_{\bm{x}}R\right)\left(\nabla_{\bm{x}}S\right)^{\top} \nonumber\\
&\hspace{2cm}+ \dfrac{\complexi}{\lambda}\left(\nabla_{\bm{x}}S\right)\left(\nabla_{\bm{x}}R\right)^{\top}\bigg\},\nonumber
\end{align*}
and so
\begin{align*}
&\langle\bm{\Sigma},\hess_{\bm{x}}\psi\rangle = \psi\bigg\{\bigg\langle\bm{\Sigma},\hess_{\bm{x}}R + \dfrac{\complexi}{\lambda}\hess_{\bm{x}}S\bigg\rangle + \|\nabla_{\bm{x}}R\|_{\bm{\Sigma}}^2 \nonumber\\ & \hspace{2.5cm} - \dfrac{1}{\lambda^2}\|\nabla_{\bm{x}}S\|_{\bm{\Sigma}}^2 + \dfrac{2\complexi}{\lambda}\left(\nabla_{\bm{x}}R\right)^{\top}\bm{\Sigma}\left(\nabla_{\bm{x}}S\right)\bigg\}. \nonumber
\end{align*}
Substituting this result back into \eqref{Weighted_LaplacianPsi} and simplifying algebraically, we obtain
\begin{align}
0=&-\dfrac{\lambda^2}{2}\Delta_{\bm{\Sigma}}\psi +\psi\bigg\{\dfrac{\lambda^2}{2}\langle\hess_{\bm{x}},\bm{\Sigma}\rangle +\dfrac{\lambda^2}{2}\langle\bm{\Sigma},\hess_{\bm{x}}R\rangle \nonumber \\ &-\dfrac{1}{2}\|\nabla_{\bm{x}}S\|_{\bm{\Sigma}}^2 +\dfrac{\lambda^2}{2}\|\nabla_{\bm{x}}R\|_{\bm{\Sigma}}^2+\lambda^2\langle\nabla_{\bm{x}}\cdot\bm{\Sigma},\nabla_{\bm{x}}R\rangle\bigg\}\nonumber\\
& +\complexi\lambda\psi\bigg\{\dfrac{1}{2}\langle\bm{\Sigma},\hess_{\bm{x}}S\rangle + \left(\nabla_{\bm{x}}R\right)^{\top}\bm{\Sigma}\left(\nabla_{\bm{x}}S\right) \nonumber \\ & \hspace{1cm}+ \langle\nabla_{\bm{x}}\cdot\bm{\Sigma},\nabla_{\bm{x}}S\rangle\bigg\}.
\label{WeightedLaplacianRHSzero}    
\end{align}
Summing \eqref{WeightedLaplacianRHSzero} and \eqref{LHSOfSchrodingerPDEgeneralized}, we arrive at \eqref{WaveFunctionPDEgeneralized}, where $\Re\left(V_{\texttt{caSB}}\right)$ and $ \Im\left(V_{\texttt{caSB}}\right)$ are as in \eqref{Re_V} and \eqref{Im_V}, respectively. By specializing \eqref{psitimespsidagger} at the initial and terminal times $t_0,t_1$, we obtain the boundary conditions \eqref{WaveFunctionBCgeneralized}.
\end{proof}

\begin{remark}\label{Remark:PDEForR}
From \eqref{partiallogpsipartialtgeneralized}, we note that the dynamics of $R$ is given by
\begin{align}
&\partial_{t}R =-\langle\nabla_{\bm{x}}R,\bm{f}+\bm{gg}^{\top}\nabla_{\bm{x}}S\rangle -\frac{1}{2}\nabla_{\bm{x}}\cdot\left(\bm{f}+\bm{gg}^{\top}\nabla_{\bm{x}}S\right) \nonumber \\ 
&+ \frac{1}{2}\Delta_{\bm{\Sigma}} R + \|\nabla_{\bm{x}}R\|_{\bm{\Sigma}}^{2} + \left(\frac{1}{4}-\frac{1}{2}R\right)\langle\hess_{\bm{x}},\bm{\Sigma}\rangle.
\label{PDEforR}    
\end{align}
\end{remark}

\begin{remark}\label{Remark:RecoveringOriginal}
Using the solution of the BVP \eqref{WaveFunctionBVPgeneralized}, one can use \eqref{psitimespsidagger} to compute the optimally controlled joint PDF at any $t\in[t_0,t_1]$. Combining \eqref{OptimalControl} and \eqref{NagasawaTransform}, the optimal control $\bm{u}_{\mathrm{opt}} = -\frac{\complexi\lambda}{2}\bm{g}^{\top}\nabla_{\bm{x}}\log\dfrac{\psi}{\psi^{\dagger}}(t,\bm{x})$, which is a real function of $(t,\bm{x})$ since $\nabla_{\bm{x}}S$ is real.
\end{remark}

Next, we discuss special cases of interest that lead to considerable simplifications of \eqref{Re_V}-\eqref{Im_V}.

\subsection{The case $\bm{\Sigma}=\lambda\bm{gg}^{\top}$}\label{subsec:ProportionalNoiseInput}
A special case that arises in practice is when the control and the noise coefficients are proportional, i.e., $\bm{\Sigma}\propto\bm{gg}^{\top}$. This case occurs when the process noise models imperfect actuation, or when the noise enters through input (e.g., force, torque, current) channels.

Notice that in Sec. \ref{subsec:psiCASB}, the $\lambda>0$ was an arbitrary constant. By interpreting $\lambda$ as the proportionality constant in $\bm{\Sigma}=\lambda\bm{gg}^{\top}$, the expressions \eqref{Re_V}-\eqref{Im_V} specialize to
\begin{align}
&\Re\left(V_{\texttt{caSB}}^{\lambda}\right) =\dfrac{\lambda^2}{2}\langle\hess_{\bm{x}},\bm{\Sigma}\rangle +\dfrac{\lambda^2}{2}\langle\bm{\Sigma},\hess_{\bm{x}}R\rangle\nonumber\\
&+\dfrac{\lambda^2}{2}\|\nabla_{\bm{x}}R\|_{\bm{\Sigma}}^2+\lambda^2\langle\nabla_{\bm{x}}\cdot\bm{\Sigma},\nabla_{\bm{x}}R\rangle + \langle\nabla_{\bm{x}}S,\bm{f}\rangle \nonumber\\
&+ \left(\dfrac{1}{2\lambda}-\dfrac{1}{2}\right)\|\nabla_{\bm{x}}S\|_{\bm{\Sigma}}^2  +\frac{1}{2}\langle\bm{\Sigma},\hess_{\bm{x}}S\rangle - q, \label{Re_V_prop}\\
&\Im\left(V_{\texttt{caSB}}^{\lambda}\right) = \dfrac{\lambda-1}{2}\langle\bm{\Sigma},\hess_{\bm{x}}S\rangle + \left( \lambda - 1 \right)\langle\nabla_{\bm{x}}R,\bm{\Sigma}\nabla_{\bm{x}}S\rangle\nonumber \\
&+ \left( \lambda - \dfrac{1}{2} \right) \langle\nabla_{\bm{x}}\cdot\bm{\Sigma},\nabla_{\bm{x}}S\rangle -\lambda\langle\nabla_{\bm{x}}R,\bm{f}\rangle
-\frac{\lambda}{2}\nabla_{\bm{x}}\cdot \bm{f} \nonumber\\
&+ \frac{\lambda}{2}\Delta_{\bm{\Sigma}} R + \lambda \|\nabla_{\bm{x}}R\|_{\bm{\Sigma}}^2 + \left(\frac{\lambda}{4}-\frac{\lambda}{2}R\right)\langle\hess_{\bm{x}},\bm{\Sigma}\rangle, \label{Im_V_prop}
\end{align}
where the superscript $\lambda$ in $V_{\texttt{caSB}}^{\lambda}$ indicates the proportionality constant.

A further special case of interest is $\lambda=1$, i.e., when the input and noise channels are the same \cite{caluya2021wasserstein,chen2021stochastic}. Then, three terms drop from \eqref{Re_V_prop}-\eqref{Im_V_prop}, giving
\begin{align}
&\Re\left(V_{\texttt{caSB}}^{1}\right) =\dfrac{1}{2}\langle\hess_{\bm{x}},\bm{\Sigma}\rangle +\dfrac{1}{2}\langle\bm{\Sigma},\hess_{\bm{x}}R\rangle\nonumber+\dfrac{1}{2}\|\nabla_{\bm{x}}R\|_{\bm{\Sigma}}^2\nonumber\\
&+\langle\nabla_{\bm{x}}\cdot\bm{\Sigma},\nabla_{\bm{x}}R\rangle + \langle\nabla_{\bm{x}}S,\bm{f}\rangle +\frac{1}{2}\langle\bm{\Sigma},\hess_{\bm{x}}S\rangle - q, \label{Re_V_lambda1}\\
&\Im\left(V_{\texttt{caSB}}^{1}\right) = \dfrac{1}{2}\langle\nabla_{\bm{x}}\cdot\bm{\Sigma},\nabla_{\bm{x}}S\rangle -\langle\nabla_{\bm{x}}R,\bm{f}\rangle
-\frac{1}{2}\nabla_{\bm{x}}\cdot \bm{f} \nonumber\\
&+ \frac{1}{2}\Delta_{\bm{\Sigma}} R + \|\nabla_{\bm{x}}R\|_{\bm{\Sigma}}^2 + \left(\frac{1}{4}-\frac{1}{2}R\right)\langle\hess_{\bm{x}},\bm{\Sigma}\rangle, \label{Im_V_lambda1}
\end{align}
wherein the superscript $1$ in $V_{\texttt{caSB}}^{1}$ indicates $\lambda=1$.

In the next Section, we further specialize \eqref{Re_V_lambda1}-\eqref{Im_V_lambda1} for the \texttt{SB}, and derive \eqref{WaveFunctionPDEclassical}. We then explain how the resulting $V_{\texttt{SB}}$ generalizes the Bohm potential.


\section{Dynamics of the Wave Function for $\texttt{SB}$}\label{sec:SpecialCase}

\subsection{Reduction to standard Schr\"{o}dinger PDE}\label{subsec:ReductionStandardSchrodingerPDE}
Recall that the \texttt{SB} corresponds to the following special case of \eqref{CASBP}: $q\equiv 0$, $\bm{f}\equiv\bm{0}$, $\bm{g}=\bm{\sigma}=\bm{I}$. These choices specialize \eqref{Re_V_lambda1}-\eqref{Im_V_lambda1} to
\begin{align}
&\Re\left(V_{\texttt{SB}}\right) = \dfrac{1}{2}\Delta_{\bm{x}}R+\dfrac{1}{2}\|\nabla_{\bm{x}}R\|^2 +\frac{1}{2}\Delta_{\bm{x}}S, \label{Re_V_SB}\\
&\Im\left(V_{\texttt{SB}}\right) = \frac{1}{2}\Delta_{\bm{x}} R + \|\nabla_{\bm{x}}R\|^2.
\label{Im_V_SB}
\end{align}
Since $\Delta_{\bm{I}}\equiv\Delta_{\bm{x}}$ (the standard Laplacian), for the potential $V_{\texttt{SB}}$ given by \eqref{Re_V_SB}-\eqref{Im_V_SB}, the PDE \eqref{WaveFunctionPDEgeneralized} reduces to \eqref{WaveFunctionPDEclassical}, which is the more familiar form of the quantum mechanical Schr\"{o}dinger PDE. The boundary conditions \eqref{WaveFunctionBCgeneralized} remain unchanged.

\begin{remark}\label{Remark:RSPDEforSB}
For the \texttt{SB}, the PDEs for $R,S$, given by \eqref{PDEforR} and \eqref{DualPDE} respectively, specialize to 
\begin{subequations}
\begin{align}
& \partial_{t} R = - \langle\nabla_{\bm{x}} R, \nabla_{\bm{x}} S\rangle -\frac{1}{2}\Delta_{\bm{x}} S + \Im\left(V_{\texttt{SB}}\right), \label{BohmPDERSB}\\
& \partial_{t} S =-\frac{1}{2}\|\nabla_{\bm{x}} S\|^2-\frac{1}{2} \Delta_{\bm{x}} S.\label{BohmPDESSB}  
\end{align}
\label{BohmPDERSSB}
\end{subequations}
\end{remark}

\subsection{Connections with Bohm theory}\label{subsec:ConnectionsWithBohm}
We now discuss the connections between \eqref{BohmPDERSSB} and the Bohm's interpretation of quantum mechanics. For a particle of mass $m>0$, Bohm's interpretation \cite{bohm1952suggested,bohm1954model} of the Schr\"{o}dinger PDE 
\begin{align}
\complexi\hslash\partial_{t}\psi = -\frac{\hslash^2}{2m}\Delta_{\bm{x}}\psi + V\psi,
\label{SchrodingerPDE}    
\end{align}
where $\hslash>0$ is the reduced Planck's constant and $V$ is a suitable potential, is ``analogous to (but not identical with) classical equations of motion" \cite{bohm1952suggested}. Similar to \eqref{NagasawaTransform}, Bohm considers 
\begin{align}
\psi = R_{\mathrm{B}}\exp(\complexi S_{\mathrm{B}}/\hslash),
\label{BohmTransform}    
\end{align} 
where the subscript ${\mathrm{B}}$ indicates Bohm, and writes the PDEs \cite[eq. (3)-(4)]{bohm1952suggested}  
\begin{subequations}
\begin{align}
& \partial_{t} R_{\mathrm{B}} =-\frac{1}{2 m}R_{\mathrm{B}} \Delta_{\bm{x}} S_{\mathrm{B}} - \frac{1}{m} \langle\nabla_{\bm{x}} R_{\mathrm{B}}, \nabla_{\bm{x}} S_{\mathrm{B}}\rangle, \label{BohmPDER}\\
& \partial_{t} S_\mathrm{B} =-\frac{\|\nabla_{\bm{x}} S_{\mathrm{B}}\|^2}{2 m}-V+\frac{\hbar^2}{2 m} \frac{\Delta_{\bm{x}} R_{\mathrm{B}}}{R_{\mathrm{B}}}.\label{BohmPDES}  
\end{align}
\label{BohmPDERS}
\end{subequations}
Comparing \eqref{BohmTransform} with \eqref{NagasawaTransform}, $R=\log R_{\mathrm{B}}$, and we can rewrite \eqref{BohmPDERS} as
\begin{subequations}
\begin{align}
& \partial_{t} R =-\frac{1}{2 m}\Delta_{\bm{x}} S_{\mathrm{B}} - \frac{1}{m} \langle\nabla_{\bm{x}} R, \nabla_{\bm{x}} S_{\mathrm{B}}\rangle, \label{BohmPDERnew}\\
& \partial_{t} S_\mathrm{B} =-\frac{\|\nabla_{\bm{x}} S_{\mathrm{B}}\|^2}{2 m}-V+\frac{\hbar^2}{2 m} \Delta_{\bm{x}} R.\label{BohmPDESnew}  
\end{align}
\label{BohmPDERSnew}
\end{subequations}
Bohm interprets the term $\frac{\hbar^2}{2 m} \Delta_{\bm{x}}R$ in the HJB PDE \eqref{BohmPDESnew} as the negative of a ``quantum potential"\footnote{which later came to be known as the Bohm potential \cite{guerra1981structural}.} that acts in addition to the ``classical potential" $V$. This allows us to view \eqref{BohmPDESnew} as the deterministic HJB PDE for a particle whose motion is governed by the sum of these two potentials, akin to classical mechanics. Bohm's parenthetical remark that the potential is not identical to the classical case refers to the nonlocal nature of this interpretation, i.e., the potential itself depends on $R$, and thus on the joint state PDF.

To seek connections among our \eqref{BohmPDERSSB} and Bohm's \eqref{BohmPDERSnew}, we identify the value function $S_{\mathrm{B}}$ with $S$, the potential $V$ with $V_{\mathrm{SB}}$, and consider normalizations $\hslash=1$, $m=1$. 

We notice that if $\Im\left(V_{\texttt{SB}}\right) = 0$, then \eqref{BohmPDERSB} matches with \eqref{BohmPDERnew}.

To relate \eqref{BohmPDESSB} with \eqref{BohmPDESnew}, notice that the aforementioned identifications allow us to rewrite \eqref{BohmPDESnew} as
\begin{align*}
\partial_{t} S &=-\frac{1}{2}\|\nabla_{\bm{x}} S\|^2-\frac{1}{2} \Delta_{\bm{x}} S - \frac{1}{2}\|\nabla_{\bm{x}}R\|^2 - \complexi\Im\left(V_{\texttt{SB}}\right).
\end{align*}
In particular, if $\Im\left(V_{\texttt{SB}}\right) = 0$, then \eqref{BohmPDESnew} becomes
\begin{align}
\partial_{t} S = -\frac{1}{2}\|\nabla_{\bm{x}} S\|^2-\frac{1}{2} \Delta_{\bm{x}} S + \frac{1}{4}\Delta_{\bm{x}}R, 
\end{align}
which is indeed \eqref{BohmPDESSB} except for the extra term $\frac{1}{4}\Delta_{\bm{x}}R$, that is a scaled ``quantum mechanical" or Bohm potential.

We have thus shown that if the potential $V_{\texttt{SB}}$ were real (i.e., elastic scattering), then the ``complexified"\footnote{Here ``complexified" simply means that the transformed variable $\psi$ in \eqref{NagasawaTransform} is a complex function of $(t,\bm{x})$. This juxtaposes with the Hopf-Cole transform \cite[Sec. IV-C]{teter2025hopf}, where the transformed variables, called the \emph{Schr\"{o}dinger factors}, are real functions of $(t,\bm{x})$.} Schr\"{o}dinger bridge admits a Bohmian interpretation. This is expected because Bohm's interpretation in \cite{bohm1952suggested} tacitly considers the $V$ in \eqref{BohmPDESnew} to be real.

Our results show that transforming the conditions of optimality for \texttt{SB} to the Schr\"{o}dinger PDE \eqref{WaveFunctionPDEclassical} is possible only if we allow the potential $V_{\texttt{SB}}$ to be complex, i.e., both the real and the imaginary parts \eqref{Re_V_SB}-\eqref{Im_V_SB} are nonzero. In this sense, $V_{\texttt{SB}}$ can be seen as a generalized Bohm potential. To the best of our knowledge, the fact that a one-to-one correspondence exists between the \texttt{SB} and the Schrödinger PDE with a complex potential has not been described before.

\begin{remark}
The use of a complex potential in Schr\"{o}dinger PDE is common in nuclear physics \cite{bethe1935theory,feshbach1958unified}, where it is called the optical potential \cite{arnold1967optical,foldy1969theory,sinha1975optical}. Here ``optical" refers to a complex refractive index whose real and imaginary parts describe the transmission and absorption of light in an optical medium, respectively. Likewise, the real part of the Schr\"{o}dinger potential encodes elastic scattering (transmission of wave function), and the imaginary part encodes inelastic scattering (absorption of wave function). 
\end{remark}

\begin{remark}
To explicitly see why $\Im(V_{\mathrm{SB}})\neq 0$, notice from \eqref{Im_V_SB} that otherwise $\frac{1}{2}\Delta_{\bm{x}} R + \|\nabla_{\bm{x}}R\|^2=0$, which by \eqref{defR}, is equivalent to $\Delta_{\bm{x}}\rho^{\bm{u}}_{\mathrm{opt}}=0$. Then, \eqref{PrimalPDE} reduces to 
\begin{align}
\partial_{t}\rho^{\bm{u}}_{\mathrm{opt}} +\nabla_{\bm{x}}\cdot\left(\rho^{\bm{u}}_{\mathrm{opt}} \bm{u}_{\mathrm{opt}}\right) = 0.
\label{uoptPDEifpossible}    
\end{align}
On the other hand, from \eqref{PrimalPDE}, the optimal control $\bm{u}_{\mathrm{opt}}^{\varepsilon}$ for the \texttt{caSB} with $q\equiv 0$, $\bm{f}\equiv\bm{0}$, $\bm{g}=\bm{\sigma}=\sqrt{\varepsilon}\bm{I}$, $\varepsilon>0$, satisfies
\begin{align}
\partial_{t}\rho^{\bm{u}^{\varepsilon}}_{\mathrm{opt}} +\nabla_{\bm{x}}\cdot\left(\rho^{\bm{u}^{\varepsilon}}_{\mathrm{opt}} \bm{u}_{\mathrm{opt}}^{\varepsilon}\right) = \frac{\varepsilon}{2}\Delta_{\bm{x}}\rho^{\bm{u}^{\varepsilon}}_{\mathrm{opt}},
\label{uepsilonoptPDE}    
\end{align}
where it is known \cite{mikami2004monge,leonard2012schrodinger} that 
$$\bm{u}_{\mathrm{opt}}^{\varepsilon} \xrightarrow{\varepsilon\downarrow 0} \bm{u}^{\varepsilon=0}_{\mathrm{opt}} \neq \bm{u}^{\varepsilon=1}_{\mathrm{opt}} =: \bm{u}_{\mathrm{opt}},$$
thereby contradicting \eqref{uoptPDEifpossible}.
\end{remark}

\section{Concluding Remarks}\label{sec:conclusions}
The purpose of this work is to clarify the connections between the control affine Schr\"{o}dinger bridge and the Schr\"{o}dinger PDE. The former is a problem in (non-quantum) stochastic control about optimal steering of probability distribution that has become a topic of significant interest in control theory and generative AI. In contrast, the Schr\"{o}dinger PDE is the foundational equation in quantum mechanics.

We show that the necessary conditions of optimality for the control affine Schr\"{o}dinger bridge problem can be transformed to suitable versions of the Schr\"{o}dinger PDE boundary value problems. Our calculations reveal that the corresponding Schr\"{o}dinger PDE must have a certain structured complex potential that is a non-trivial generalization of the Bohm potential in quantum mechanics. Like the original Bohm potential, the derived generalization is nonlocal. Our main contribution is to demonstrate that the non-equilibrium statistical mechanics of the control affine Schr\"{o}dinger bridge problem induces an absorbing medium for the wave function that manifests as a nonzero imaginary part of the potential. 
While the focus of our study was theoretical, computational implication of our results, i.e., solving density steering by wave steering, remains to be explored. 
This will be pursued in our future work.


\bibliographystyle{IEEEtran}
\bibliography{References.bib}

\end{document}